\documentclass[10pt]{amsart}

\usepackage{latexsym, amsmath,amssymb}
\usepackage{color}

\renewcommand{\epsilon}{\varepsilon}
\newcommand{\newsection}[1]
{\subsection{#1}\setcounter{theorem}{0} \setcounter{equation}{0}
\par\noindent}

\newtheorem{theorem}{Theorem}

\newtheorem{lemma}[theorem]{Lemma}
\newtheorem{corr}[theorem]{Corollary}

\newtheorem{proposition}[theorem]{Proposition}
\newtheorem{deff}[theorem]{Definition}

\newcommand{\bth}{\begin{theorem}}
\newcommand{\ble}{\begin{lemma}}
\newcommand{\bcor}{\begin{corr}}

\newcommand{\bdeff}{\begin{deff}}

\newcommand{\bprop}{\begin{proposition}}
\newcommand{\ele}{\end{lemma}}
\newcommand{\ecor}{\end{corr}}
\newcommand{\edeff}{\end{deff}}

\newcommand{\eprop}{\end{proposition}}

\newcommand{\supp}{\text{supp}}
\renewcommand{\Pi}{\varPi}
\renewcommand{\Re}{{\rm{Re} }\,}

\renewcommand{\epsilon}{\varepsilon}

\newcommand{\R}{{\mathbb R}}
\newcommand{\la}{{\langle}}
\newcommand{\ra}{{\rangle}}
\newcommand{\cd}{{\,\cdot\,}}

\renewcommand{\S}{{\mathbb{S}}}

\newcommand{\N}{{\mathbb{N}}}
\newcommand{\g}{{\mathfrak{g}}}
\newcommand{\ang}{{\not\negmedspace\nabla}}
\newcommand{\angdelta}{{\not\negthickspace\Delta}}

\begin{document}
\bibliographystyle{plain}

\title
{
On the interaction of metric trapping and a boundary
}

\author{Kiril Datchev}
\address{Department of Mathematics, Purdue University, West Lafayette,
  IN 47907-2067}
\email{kdatchev@purdue.edu}

\author{Jason Metcalfe}
\address{Department of Mathematics, University of North Carolina, Chapel Hill}
\email{metcalfe@email.unc.edu}

\author{Jacob Shapiro}
\address{Mathematical Sciences Institute, Australian National
  University, Acton, ACT 2601, Australia}
\address{Department of Mathematics, University of Dayton, Dayton, OH 45469-2316}

\email{jshapiro1@udayton.edu}

\author{Mihai Tohaneanu}
\address{Department of Mathematics, University of Kentucky, Lexington,
  KY 40506-0027}
\email{mihai.tohaneanu@uky.edu}

\thanks{K. Datchev was supported in part by NSF grant DMS-1708511,
  J. Shapiro was supported in part by the Australian Research Council
  through grant DP180100589, and M. Tohaneanu was supported in part by
Simons Collaboration Grant 586051}

\begin{abstract}
By considering a two ended warped product manifold, we demonstrate a
bifurcation that can occur when metric trapping interacts with a
boundary.  In this highly symmetric example, as the boundary passes
through the trapped set, one goes from a nontrapping scenario where
lossless local energy estimates are available for the wave equation to
the case of stably trapped rays where all but a logarithmic amount
of decay is lost.
\end{abstract}

\maketitle


\newsection{Introduction}
We explore the interaction of metric trapping and a boundary in an
explicit example and note an extreme bifurcation in the behavior of
the local energy for the wave equation as the boundary passes through
the trapping.  This is closely related to the instability of
ultracompact neutron stars as was examined in \cite{Keir}.  Here, we
instead examine a certain class of exterior domains with Dirichlet
boundary conditions on a warped product background geometry and
provide a more elementary argument.

For the Minkowski wave equation $\Box = \partial_t^2-\Delta$, we have
a conserved energy $E_0[u](t) = \frac{1}{2}\|\partial u(t,\cd)\|^2_{L^2}$ for
solutions to the homogeneous wave equation.  Here
$\partial u = (\partial_t u,\nabla_x u)$ denotes the space-time
derivative.  One common and robust measure of dispersion is called the (integrated)
local energy estimate, which involves examining the energy within a
compact set.  Specifically if we set
\[\|u\|_{LE[0,T]}=\sup_{j\ge 0}2^{-j/2} \|u\|_{L^2_tL^2_x([0,T]\times\{\la
    x\ra\approx 2^j\})},\quad \|u\|_{LE^1[0,T]}=\|(\partial u, \la
  x\ra^{-1}u)\|_{LE[0,T]}\]
and
\[\|F\|_{LE^*[0,T]} = \sum_{j\ge 0} 2^{j/2} \|F\|_{L^2_tL^2_x([0,T]\times\{\la
    x\ra\approx 2^j\})},\]
we have
\[\|\partial u\|_{L^\infty_t L^2_x} + \|u\|_{LE^1} \lesssim \|\partial
  u(0,\cd)\|_{L^2} + \|\Box u\|_{L^1_tL^2_x + LE^*}\]
on $(1+3)-$dimensional Minkowski space.  Here $LE^1$ and $LE^*$ are understood to
denote $LE^1[0,\infty)$, $LE^*[0,\infty)$.  Such estimates originated in
\cite{Mora1, Mora2}.  See, e.g., \cite{MST} for some of the
most general results and a more complete history.

On nonflat geometries, the null geodesics, which packets of the
solution tend to flow along, are no longer necessarily straight lines.  And in
certain geometries, null geodesics may stay in a compact set for all
times, and when this happens, trapping is said to occur.  Trapping is
a known obstruction to local energy estimates.  In fact,
\cite{Ralston}, \cite{Sbierski} show that the local energy estimate as
 stated above cannot hold when trapped rays exist.

When the trapping is sufficiently unstable, it is often the case that
local energy estimates may be recovered with a small loss, which is
often realized as a loss of regularity.  This is what happens, e.g.,
on the Schwarzschild space-times \cite{MMTT_Str_BH}.  There it is shown that
a logarithmic loss of regularity suffices.  See \cite{Ik_two, Ik_several} for the seminal
results in this direction.  
When, however, the trapping is elliptic
(i.e. stable), it is known that nearly all decay is lost.  See, e.g.,
\cite{Burq_Fr}, \cite{HS}, \cite{Keir}.  In both of these cases,
numerous related results have followed.  See, e.g., \cite[Chapter 6]{DZ_book}.
The surfaces that we consider
are from \cite{CW}, \cite{BCMP} where they were shown to generate an
example of trapping for which an algebraic loss of regularity is both
necessary and sufficient.

The notion of being nontrapping is generally known to be stable in the
sense that a sufficiently small perturbation of a nontrapping metric
remains nontrapping.  Here, however, we show that a drastic
bifurcation can happen when metric trapping interacts with a
boundary.  Specifically, on the surfaces used in \cite{CW},
\cite{BCMP}, we shall demonstrate that lossless local energy estimates
are available when a boundary exists on one side of the trapping.  But
as soon as that boundary passes through the trapped set, the
interaction with the metric causes stable trapping to form.
In this setting, we
demonstrate that at most a logarithmic amount of decay is available
and no loss of regularity is sufficient in order to recover a local
energy estimate.

Specifically we consider the warped products that were examined in
\cite{CW}, \cite{BCMP}.  That is, we examine $\R\times\R\times\S^2$
with 
\[ds^2=-dt^2+dx^2+a(x)^2\,d\sigma_{\S^2}^2,\quad
  a(x)=(x^{2m}+1)^{1/2m},\quad m\in\N.\]
Here $(\S^2,
d\sigma^2_{\S^2})$ is the two-dimensional round sphere.
This geometry is asymptotically flat on both of its ends, and trapping occurs
at $x=0$.  When $m\ge 2$, the trapping is degenerate, while for $m=1$
the trapping resembles that of the Schwarzschild metric.  We use $M_{x_0}$ to denote the space $(x_0,\infty)\times \S^2$ equipped with the metric
$dx^2+a(x)^2d\sigma^2_{\S^2}$.  
We will set $dV=a(x)^2\,dx\,d\sigma_{\S^2}$, while
the volume form of the whole space-time $\R_+\times M_{x_0}$, then, is $dV\,dt$. The arguments of this paper should also apply if $\S^2$ is replaced by other compact manifolds, but we will not pursue it here.

On this background (and in these coordinates), the wave equation is given by
\[\Box_\g u = - \partial_t^2u + \Delta_{\g}u= - \partial_t^2 u+
  a(x)^{-2}\partial_x\Bigl[a(x)^2\partial_x  u\Bigr] +
  a(x)^{-2} \angdelta_{\S^2} u.\]
We consider the boundary value problem with Dirichlet boundary conditions
\begin{equation}
  \label{bvp}
  \begin{split}
    \Box_\g u &= F(t,x,\omega),\quad\qquad\qquad (t,x,\omega)\in \R_+\times \{x\ge x_0\} \times
    \S^2,\\
    u(t,x_0,\omega)&=0,\\
    u(0,x,\omega)=u_0(x,\omega)&,\quad \partial_t u(0,x,\omega)=u_1(x,\omega).
  \end{split}
\end{equation}

Our methods do not directly apply for other boundary conditions. For example, in the proof of Theorem~\ref{PositiveThm} for Neumann boundary conditions, there will be an extra boundary term on $x=x_0$ with the wrong sign.

This static space-time and the Dirichlet boundary conditions naturally
yield a coercive conserved energy for solutions to the homogeneous
equation ($F\equiv 0$)
\[E[u](t) = \frac{1}{2} \int_{x\ge x_0} (\partial_t u)^2 + (\partial_x u)^2 +
  a(x)^{-2} |\ang_0 u|^2\,dV,\]
where $\ang_0$ denotes the derivatives tangent to the unit sphere.
More generally, we have
\begin{equation}
  \label{energy}
  E[u](t) \le E[u](0) + \int_0^t \int_{x\ge x_0} |\Box_\g u|
  |\partial_t u|\,dV\,dt.
\end{equation}

We first consider the case of $x_0>0$.  In this realm, the trapping is
not observed and the effect of the boundary is akin to the case of
star-shaped obstacles as was examined in \cite{MS_SIAM}.  We note
that, in this case, we can significantly simplify the argument of  \cite{BCMP} and indeed select a single
multiplier that will yield the result rather than needing to consider
high and low frequency regimes separately.  

\begin{theorem}\label{PositiveThm}
If $x_0>0$, 
then solutions to
the wave equation \eqref{bvp} satisfy the lossless local
energy estimate\footnote{The analog of the $LE$ norm here is
  \[\|u\|_{LE[0,T]} = \sup_{j\ge 0} 2^{-j/2} \Bigl(\int_0^T \int_{\{\la
    x\ra \approx 2^j\}\cap \{x\ge x_0\}} \int_{\S^2} |u(t,x,\omega)|^2
  a(x)^2\,d\sigma(\omega)\,dx\,dt\Bigr)^{1/2},\]
with similar adjustments for $LE^*$.}
\begin{equation}
  \label{LEpositive}
  \|u\|^2_{LE^1} + \sup_t E[u](t) \lesssim E[u](0) + \|\Box_\g
  u\|^2_{L^1_tL^2_{M_{x_0}}+LE^*}.
\end{equation}
\end{theorem}

For each $R>0$, we shall consider the local energy
\[E_{R}[u](t) = \frac{1}{2}\int_{x_0\le x\le R} (\partial_t u)^2 + (\partial_x
  u)^2 + a(x)^{-2} |\ang_0 u|^2\,dV.\]
We shall use
\[\|(u_0,u_1)\|_{D(B^k)} := \|(u_0,u_1)\|_{H_{x_0}} +
  \|B^k(u_0,u_1)\|_{H_{x_0}}\]
with
\[H_{x_0} := \dot{H}_0^1(M_{x_0}) \oplus L^2(M_{x_0}),\quad B:=
  \begin{bmatrix}
    0&iI\\i\Delta_{\g}&0
  \end{bmatrix}.\]

Our second theorem then says that when $x_0<0$ all but a logarithmic
amount of decay is lost no matter what loss of regularity $k$ is
permitted.  
\begin{theorem}\label{NegativeThm}
Let $x_0<0$, and fix $R>0$.  Then for any $k\in \N$, if $u$ solves
\eqref{bvp} with $F\equiv 0$, 
\begin{equation}
  \label{1.3}
  \lim\sup_{t\to \infty} \Bigl(\log^k(t) \sup_{u_0,u_1} \frac{E^{1/2}_R[u](t)}{\|(u_0,u_1)\|_{D(B^k)}}\Bigr)>0,
\end{equation}
where the supremum is taken over all $u_0,u_1\in C^\infty(M_{x_0})$
supported in $\{x_0\le x<R\}$ that vanish when $x=x_0$.
\end{theorem}
We
note that the lower bound that we obtain here matches up nicely with
the decay obtained in \cite[Th\'eor\`eme 1.1]{Burq_Fr}, namely
\[E_R^{1/2}[u](t)\lesssim
  \frac{\|(u_0,u_1)\|_{D(B^k)}}{\log^k(t+2)},\quad t\ge 0.\]
Note, however, that the assumptions in \cite{Burq_Fr} are not exactly
the same as ours, requiring in particular $a(x)=x$ for $x\gg 1$. While we expect a similar result to hold in our context, we do not prove it here.

We also remark that there is no reason that $R>0$ is required.  The
choice was made simply to reduce the number of parameters for the sake
of clarity.

Due to the presence of stably trapped broken
bicharacteristics, the integrated local energy estimate must also
fail.  While the above morally states that the solution decays at most
logarithmically, it does not strictly rule out the possibility of an
integrated local energy estimate.  Thus, we also prove the following.
\begin{theorem}\label{NegativeLEThm}
 For any $A>0$ and any $k\in \N$, there exists a $T>0$ and data
$u_0,u_1\in C^\infty(M_{x_0})$, which are 
supported in $\{x_0\le x<R\}$ and vanish when $x=x_0$,
so that the solution $u$ to \eqref{bvp} when $F\equiv 0$ satisfies
 \begin{equation}
   \label{negativele}
   \|u\|_{LE^1[0,T]} > A \|(u_0,u_1)\|_{D(B^k)}.
 \end{equation}
\end{theorem}

\newsection{Proof of Theorem~\ref{PositiveThm}: Nontrapping with a star-shaped boundary}
By standard approximation arguments, we may assume that $u_0$, $u_1$,
and $F$ have spatial support inside a fixed ball.  Finite speed of
propation implies that $u(t,x,\omega)$ has compact support in $x$ for
any $t$.

For, say, $f, g\in C^2$, integration by parts and the Dirichlet
boundary conditions give
\begin{multline}
  \label{ibp}
  -\int_0^T\int_{x\ge x_0} \Box_\g u \Bigl[f(x)\partial_x u + g(x)
  u\Bigr]\,dV\,dt = \int_{x\ge x_0} \partial_t u \Bigl[f(x)\partial_x u + g(x)
  u\Bigr]\,dV\Bigl|_0^T
\\+ \int_0^T\int_{x\ge x_0} \Bigl[f'(x)+g(x)-\frac{1}{2}a(x)^{-2}\frac{d}{dx}
\Bigl(a(x)^2 f(x)\Bigr)\Bigr] (\partial_x u)^2\,dV\,dt
\\+\int_0^T\int_{x\ge x_0} \Bigl[f(x)\frac{a'(x)}{a(x)}+g(x)-\frac{1}{2}a(x)^{-2}\frac{d}{dx}
\Bigl(a(x)^2 f(x)\Bigr)\Bigr] a(x)^{-2}|\ang_0 u|^2\,dV\,dt
\\+\int_0^T \int_{x\ge x_0} \Bigl[-g(x)+\frac{1}{2}a(x)^{-2}\frac{d}{dx}
\Bigl(a(x)^2 f(x)\Bigr)\Bigr] (\partial_t u)^2\,dV\,dt
\\-\frac{1}{2}\int_0^T\int_{x\ge x_0} \Bigl(a(x)^{-2}\frac{d}{dx}
[a(x)^2 g'(x)]\Bigr) u^2\,dV\:dt +
\frac{1}{2}\int_0^T\int_{\S^2} f(x_0) (\partial_x u)^2|_{x=x_0} a(x_0)^2\,d\sigma_{\S^2}dt.
\end{multline}
The identity \eqref{ibp} can alternatively be seen by applying the
Fundamental Theorem of Calculus to
\[\int_0^T \int_{x\ge x_0} \Bigl(\partial_t I_1 + \partial_x I_2 +
  \ang_0\cdot I_3\Bigr)\,d\sigma_{\S^2}\,dx\,dt\]
where
\begin{align*}I_1 &= - \partial_t u \Bigl(f(x)\partial_x u + g(x) u\Bigr)
  a(x)^2,\\
I_2&=\frac{1}{2}\Bigl((\partial_t u)^2 + (\partial_x u)^2 -
     a(x)^{-2}|\ang_0 u|^2\Bigr)f(x)a(x)^2 + u\partial_x u g(x) a(x)^2
     - \frac{1}{2}g'(x)u^2 a(x)^2,\\
I_3&=\Bigl(f(x)\partial_x u + g(x) u\Bigr)\ang_0 u.
\end{align*}

For $0<\delta\ll 1$, we 
set 
\[f(x)=\frac{x^2}{a(x)^2},\quad g(x) =
\frac{1}{2}a(x)^{-2}\frac{d}{dx}\Bigl(a(x)^2f(x)\Bigr) - \delta \frac{x^{1+2m}}{(1+x^{2m})^{2+\frac{1}{m}}}.\]
Then we record that, for $\delta$ sufficiently small, the coefficient of $(\partial_x u)^2$ satisfies
\begin{align*}f'(x)+g(x)-\frac{1}{2}a(x)^{-2}\frac{d}{dx}\Bigl(a(x)^2f(x)\Bigr) &=
\frac{2x}{(1+x^{2m})^{1+\frac{1}{m}}}
- \delta \frac{x^{1+2m}}{(1+x^{2m})^{2+\frac{1}{m}}}\\
&\gtrsim \frac{x}{(1+x^{2m})^{1+\frac{1}{m}}},
\end{align*}
the coefficient of $a(x)^{-2}|\ang_0 u|^2$ becomes
\begin{align*} f(x)\frac{a'(x)}{a(x)} + g(x) -\frac{1}{2}a(x)^{-2} \frac{d}{dx}
\Bigl(a(x)^2f(x)\Bigr) &= \frac{x^{1+2m}}{(1+x^{2m})^{1+\frac{1}{m}}}-
                         \delta
                         \frac{x^{1+2m}}{(1+x^{2m})^{2+\frac{1}{m}}} 
  \\
&\gtrsim \frac{x^{1+2m}}{(1+x^{2m})^{1+\frac{1}{m}}},
\end{align*}
the coefficient of $(\partial_t u)^2$ is
\[-g(x)+\frac{1}{2}a(x)^{-2}\frac{d}{dx}
\Bigl(a(x)^2 f(x)\Bigr) = 
\delta \frac{x^{1+2m}}{(1+x^{2m})^{2+\frac{1}{m}}},\]
and the coefficient of $u^2$ obeys
\begin{align*}-\frac{1}{2}a(x)^{-2}\frac{d}{dx}
[a(x)^2 g'(x)] &=
  \frac{2m\:x^{2m-1}}{(1+x^{2m})^{2+\frac{1}{m}}} + \delta
                        \frac{m(2m+1) x^{2m-1}
                        (x^{4m}-4x^{2m}+1)}{(1+x^{2m})^{4+\frac{1}{m}}}\\&\gtrsim
                                                                           \frac{x^{2m-1}}{(1+x^{2m})^{2+\frac{1}{m}}}  
\end{align*}
provided that $\delta>0$ is sufficiently small.
Moreover, we note that the boundary term, which is the last term in
\eqref{ibp}, is nonnegative.

Since $0\le f(x)\le 1$, the Schwarz inequality and
\eqref{energy} give
\[\Bigl| \int_{x\ge x_0} f(x) \partial_t u \partial_x u \,dV\Bigr| \lesssim
E[u](t) \le E[u](0) + \int_0^t\int_{x\ge x_0} |\Box_\g u||\partial u|\,dV\,dt\]
on any time slice.  Similarly, since $g(x)\lesssim a(x)^{-1}$, we have
\[\Bigl|\int_{x\ge x_0} g(x) u\partial_t u\, dV\Bigr|\lesssim
\Bigl(\int_{x\ge x_0} a(x)^{-2} u^2\, dV\Bigr)^{1/2}
(E[u](t))^{1/2}.\]
So upon establishing a Hardy-type inequality
\begin{equation}
  \label{hardy}
  \int_{x\ge x_0} a(x)^{-2} u^2\,dV \lesssim \int_{x\ge x_0}
  (\partial_x u)^2\, dV,
\end{equation}
we shall also have
\[\Bigl| \int_{x\ge x_0} g(x) u \partial_t u \,dV\Bigr| 
\lesssim E[u](0) + \int_0^t\int_{x\ge x_0} |\Box_\g u| |\partial u|\,dV\,dt\]
on every time slice.  In order to prove \eqref{hardy}, 
we simply note that $x \le a(x)$ and integrate by parts, while
    relying on the Dirichlet boundary conditions, to obtain
\begin{align*}\int_{x\ge x_0} a(x)^{-2} u^2 \,dV &= \int_{x\ge x_0} u^2 \frac{d}{dx}
x\,dx\,d\sigma\\ &\lesssim\int_{x\ge x_0} a(x)^{-1} |u \partial_x u| \,
dV\lesssim \|a(x)^{-1} u\|_{L^2} \|\partial_x u\|_{L^2}.
\end{align*}

Using the bounds from below for each of the coefficients in
\eqref{ibp} and the above estimation of the time-boundary terms, a
local energy estimate of the form
\begin{multline}\label{LElocal}
\int_0^T \int_{x\ge x_0} \Bigl(x^{-2m-1} (\partial_x u)^2 + x^{-1} a(x)^{-2}
|\ang_0 u|^2 + x^{-2m -1}(\partial_tu)^2 + x^{-2m-3} u^2\Bigr)
\,dV\,dt\\\lesssim E[u](0) + \int_0^T\int_{x\ge x_0} |\Box_\g u|\Bigl(|\partial u|+a(x)^{-1}|u|\Bigr)\,dV\,dt
\end{multline}
follows, though it does not have the sharp
weights as $x\to \infty$.  

To get the estimate as stated, we shall
pair \eqref{LElocal} with \cite[Proposition 2.3]{BCMP}, which showed 
\begin{equation}
\label{BCMPestimate}
\|u\|^2_{LE^1_{x>R}[0,T]} \lesssim E[u](0) + \int_0^T\int_{x\ge x_0} |\Box_\g
u|\Bigl(|\partial u|+ a(x)^{-1}|u|\Bigr)\,dV\,dt+ R^{-2}
\|u\|^2_{LE_{x\approx R}[0,T]} 
\end{equation}
provided $R$ is sufficiently large.  Here, e.g., $LE^1_{x>R}[0,T]$
denotes the $LE^1[0,T]$ norm restricted to the set $\{(t,x,\omega)\,:\,
0\le t\le T, \,
x>R\}$.  In order to prove \eqref{BCMPestimate}, we again use
\eqref{ibp} but this time with
\[f(x)=\Bigl(1-\beta\Bigl(\frac{x}{R}\Bigr)\Bigr)\frac{x}{x+\rho},\quad
  g(x)=\frac{1}{2}a(x)^{-2}\frac{x}{x+\rho}\frac{d}{dx}\Bigl[\Bigl(1-\beta\Bigl(\frac{x}{R}\Bigr)\Bigr)a(x)^2\Bigr],\quad
  \rho\ge R\]
where $\beta(\rho)$ is a monotonically decreasing cutoff that is
$\equiv 1$ for $\rho<1/2$ and vanishes for $\rho>1$.  See \cite{BCMP}
for more details.

Due to \eqref{BCMPestimate}, it suffices to control $\|u\|^2_{LE^1_{x_0\le x\le R}}$
for which the weights at infinity are irrelevant and \eqref{LElocal}
suffices.  Applying the Schwarz inequality to the forcing term and
bootstrapping completes the proof.

\newsection{Proof of Theorem~\ref{NegativeThm} and Theorem~\ref{NegativeLEThm}: Stable trapping and the
  nonexistence of integrable local energy decay}  Here we shall need
the following sequence of exponentially good quasimodes.
\begin{proposition}
  \label{prop_quasimode}
There is a constant $c>0$, a sequence of positive numbers $\tau_j\to
+\infty$, and functions $v_j\in C^\infty(M_{x_0})$ such that $\supp\,
v_j\subseteq \{x_0\le x<0\}$, $\|v_j\|_{L^2(M_{x_0})}=1$, $v_j|_{x=x_0}=0$, and for each
  $k=0,1,2,\dots$ there exists $C_k>0$ so that
  \begin{equation}
    \label{quasimode}
    \|(-\Delta_{\g} -\tau_j^2)v_j\|_{H^k(M_{x_0})} \le C_k e^{-c\tau_j}.
  \end{equation}
\end{proposition}

Before we proceed to proving the Proposition, we shall first show how
these quasimodes can be used to complete the proof of
Theorem~\ref{NegativeThm}.  These arguments are akin to those of
\cite{Ralston}, \cite{HS}, and others.

Let 
\[u_{0,j}(x,\omega) := v_j(x,\omega),\quad
  u_{1,j}(x,\omega):=-i\tau_j v_j(x,\omega),\]
\[U_j(t) = (U_{0,j}(t), U_{1,j}(t)) :=
  e^{-it\tau_j}(v_j,-i\tau_j v_j)\in H_{x_0}.\]
It follows immediately from the definition of $\|\,\cdot\,\|_{D(B^k)}$
and Proposition~\ref{prop_quasimode} that for each $k\in \N$, there is
$C_k>0$ so that for all $j\in \N$
\begin{equation}
  \label{3.1}
  \|(u_{0,j},u_{1,j})\|_{D(B^k)} =\|(v_j, -i\tau_j v_j)\|_{D(B^k)} \le
  C_k \tau^k_j \|(v_j,-i\tau_j v_j)\|_{H_{x_0}}.
\end{equation}

One can check that $U_j$ solves the inhomogeneous wave equation
\[
  \begin{cases}
    \partial_t U_j + iBU_j = F_j :=(0,
    (-\Delta_{\g}-\tau_j^2)v_j),\quad \text{ on } \R_+\times
    M_{x_0},\\
U_j(0)=(v_j, -i\tau_j v_j).
  \end{cases}
\]
Next let
\[\tilde{U}_j(t)=(\tilde{U}_{0,j}(t),\tilde{U}_{1,j}(t)) :=
  e^{-itB}(v_j,-i\tau_j v_j)\]
be the solution to the homogeneous equation where $F_j=0$.  Note that
if $u_j$ solves \eqref{bvp} with $u_0=u_{0,j}, u_1=u_{1,j}$, and
$F\equiv 0$, then
$\tilde{U}_{0,j}=u_j$, $\tilde{U}_{1,j}=\partial_t u_j$.  By
Duhamel's principle, the relationship between $U_j$ and $\tilde{U}_j$
is
\[U_j(t) = \tilde{U}_j(t) + \int_0^t e^{-i(t-s)B} F_j\,ds.\]
Using \eqref{quasimode}, we estimate
\begin{align*}
 \Bigl(\int_{x_0\le x\le R} |\nabla_\g (U_{0,j}-\tilde{U}_{0,j})(t)|^2 +
  |(U_{1,j}-\tilde{U}_{1,j})(t)|^2\,dV\Bigr)^{1/2}&\le \Bigl\|\int_0^t
                                                    e^{-i(t-s)B}F_j\,ds\Bigr\|_{H_{x_0}}\\
&\le t\|F_j\|_{H_{x_0}} \le C_0 te^{-c\tau_j}.
\end{align*}
Let $J\in \N$ be sufficiently large that $\tau_j\ge 1$ for $j\ge J$.
Then for any $t\in [0,t_j]$ where
\[t_j:=\frac{1}{2C_0} e^{c\tau_j} \le \frac{1}{2C_0}e^{c\tau_j} \|(v_j,-i\tau_j v_j)\|_{H_{x_0}},\]
it holds that
\begin{align*}
  E^{1/2}_R[u_j](t)&\ge \|U_j(t)\|_{H_{x_0}} - \Bigl(\int_{x_0\le x\le R}
  |\nabla_\g (U_{0,j}-\tilde{U}_{0,j})(t)|^2 +
  |(U_{1,j}-\tilde{U}_{1,j})(t)|^2\,dV\Bigr)^{1/2}\\
&\ge \|(v_j, -i\tau_j v_j)\|_{H_{x_0}} - C_0t e^{-c\tau_j}\\
&\ge \frac{1}{2} \|(v_j,-i\tau_j v_j)\|_{H_{x_0}}\\
&\ge \frac{1}{2C_k \tau_j^k} \|(v_j,-i\tau_j v_j)\|_{D(B^k)}.
\end{align*}
In the last step, we have used \eqref{3.1}.  We have thus shown
\[\frac{E^{1/2}_{R}[u_j](t)}{\|(u_{0,j},u_{1,j})\|_{D(B^k)}} \ge
    \frac{1}{2C_k\tau_j^k} = \frac{c^k}{2C_k} \log^{-k}(2C_0t_j),\quad
    t\in [0,t_j],\]
from which Theorem~\ref{NegativeThm} follows immediately.  

By integrating the above inequality, we also obtain
\begin{align*}
  \|u_j\|_{LE^1[0,t_j]} 
&\ge C_{x_0} \|E_0^{1/2}[u_j](t)\|_{L^2[0,t_j]} \\
&\ge C_{x_0,k} \frac{t_j^{1/2}}{\log^k(t_j)} \|(u_{0,j},u_{1,j})\|_{D(B^k)}.
\end{align*}
%
Since $\tau_j\to \infty$ as $j\to\infty$, given any $A>0$, we can select
$j$ sufficiently large so $C_{x_0,k} \frac{t_j^{1/2}}{\log^k(t_j)} >
A$, which proves Theorem~\ref{NegativeLEThm}.

\subsubsection{Proof of Proposition~\ref{prop_quasimode}:}
By expanding into spherical harmonics, we can write
\[a(x)(-\Delta_\g)a(x)^{-1} = \bigoplus_{j=0}^\infty
  \Bigl(-\frac{d^2}{dx^2} + \sigma_j^2 a(x)^{-2} + a''(x)a(x)^{-1}\Bigr)\]
where $0=\sigma_0 < \sigma_1\le
\sigma_2\le \dots$ are the square roots of the eigenvalues of the
Laplacian on $\S^2$, repeated according to multiplicity.  We set
\[V_j(x):=\sigma_j^2 a(x)^{-2}+a''(x)a(x)^{-1},\quad
  V_0:=a''(x)a(x)^{-1}.\]

We will show that there is a sequence $\tau_j\to +\infty$ so that we
have $u_j\in C_c^\infty([x_0,0))$, $\|u_j\|_{L^2(\R)}=1$, $u(x_0)=0$, and for each $k=0,1,2,\dots$
\begin{equation}
  \label{quasimode_separated}
  \Bigl(-\frac{d^2}{dx^2} + V_j(x)-\tau_j^2\Bigr)u_j(x)=\mathcal{O}_{H^k((x_0,0))}(e^{-c\tau_j}).
\end{equation}
This will imply \eqref{quasimode} with $v_j(x,\omega) =
a(x)^{-1}u_j(x) Y_j(\omega)$ where $Y_j$ is a spherical harmonic
with eigenvalue $\sigma_j^2$.

The first lemma fixes the sequence $\tau_j$.
\begin{lemma}
  \label{lemma_evalues}
For $j$ large enough, $V_j$ is strictly increasing on $[x_0,x_0/2]$,
we have $V_j(x_0/2) < V_j(0)$, and the operator $P:=-\frac{d^2}{dx^2} + V_j(x)$ on
$(x_0,0)$ with Dirichlet boundary conditions has an eigenvalue
$\tau_j^2\in [V_j(x_0), V_j(x_0/2)]$.
\end{lemma}

\begin{proof} 
To prove that, for $j$ large enough, we have $V_j$ strictly increasing
on $[x_0, x_0/2]$ and $V_j(x_0/2)<V_j(0)$, it suffices to use the fact
that $\sigma_j^2 a(x)^{-2}$ is strictly increasing on $[x_0,0]$ when
$j>0$ and that $\sigma_j\to \infty$ as $j\to \infty$.

Let
$I=(x_0,0)$, let $\mathcal{D}=H^1_0(I)\cap H^2(I)$ 
be the domain of $P$, and
 let $\tau_j^2$ be the first eigenvalue of $P$.
We prove the upper bound on $\tau_j^2$ by comparing the bottom of the spectrum of
$P$ with the bottom of the spectrum of an explicit infinite
square well, whose first eigenvalue can be directly calculated.  Let
\[P_+ = -\frac{d^2}{dx^2} + V_j(3x_0/4),\quad \text{have domain }
  \mathcal{D}_+=H^1_0(I_+)\cap H^2(I_+)\]
where $I_+=(x_0,3x_0/4)$.
  We then have
 \begin{align*}
   \tau^2_j = \inf_{u\in \mathcal{D}} \frac{\la Pu,
   u\ra_{L^2(I)}}{\|u\|^2_{L^2(I)}} &\le \inf_{u\in C^\infty_0(I_+)}                                                   
 \frac{\la Pu,
  u\ra_{L^2(I)}}{\|u\|^2_{L^2(I)}}
\\
 &\le \inf_{u\in C^\infty_0(I_+)} \frac{\la P_+u,
  u\ra_{L^2(I_+)}}{\|u\|^2_{L^2(I_+)}}\\
&=V_j(3x_0/4) + \inf_{u\in C^\infty_0(I_+)}
  \frac{\|u'\|^2_{L^2(I_+)}}{\|u\|^2_{L^2(I_+)}} = V_j(3x_0/4)+16\pi^2 x_0^{-2},
 \end{align*}
which is bounded by $V_j(x_0/2)$ for $j$ large enough.  The last
equality follows by computing the smallest constant $\rho$ so
$u''+\rho u=0$ has a nontrivial solution with $u(x_0)=u(3x_0/4)=0$.

To prove the remaining lower bound, we observe that
\[\inf_{u\in \mathcal{D}} \frac{\la
    Pu,u\ra_{L^2(I)}}{\|u\|^2_{L^2(I)}} \ge V_j(x_0).\]\end{proof}

Fix $\chi\in C^\infty([x_0,0]; [0,1])$ with $\chi(x)\equiv 1$ on a
neighborhood of $[x_0,x_0/2]$ and $\chi(x)\equiv 0$ on a neighborhood
of $0$.  Let $\psi_j\in C^\infty(I)$, $\psi_j(x_0)=\psi_j(0)=0$ be
an eigenfunction associated to the eigenvalue $\tau_j^2$, as supplied
by Lemma~\ref{lemma_evalues}.  Thus,
\begin{equation}
  \label{2.3}
  \Bigl(-\frac{d^2}{dx^2} + V_j(x)-\tau^2_j\Bigr)\psi_j = 0,\quad \tau_j^2 \in
  [V_j(x_0), V_j(x_0/2)].
\end{equation}
We define the quasimodes to be
\[u_j(x) = \chi(x) \psi_j(x) / \|\chi \psi_j\|_{L^2}.\]

In order to prove \eqref{quasimode_separated}, we will prove the
following Agmon estimate, which is a variant of the standard
semiclassical Agmon estimate as in \cite[Section 7.1]{Zworski}.
\begin{lemma}\label{lemma_agmon}
 There exists a constant $c$ so that, for $j$ large enough,
 \begin{equation}
   \label{agmon}
   \|\mathbf{1}_{\emph{supp}\,(1-\chi)} \psi_j\|_{L^2([x_0,0])} \le
   e^{-c\sigma_j} \|\psi_j\|_{L^2([x_0,0])}.
 \end{equation}
\end{lemma}

\begin{proof}
 Let $\varphi\in C^\infty([x_0,0])$ such that $\varphi \equiv 0$ on
 a neighborhood of $[x_0, x_0/2]$ and $\varphi\equiv 1$ on a neighborhood of $\supp\,
 (1-\chi)$.  
We then fix $\chi_0\in C^\infty([x_0,0])$ with
$\chi_0\equiv 0$ on a neighborhood of $[x_0, x_0/2]$ and $\chi_0\equiv 1$ on a neighborhood of
$\supp\,\varphi$.  Define
\begin{equation}
  \label{2.6}
  w:=\chi_0 e^{\delta \sigma_j \varphi} \psi_j,\quad \delta>0 \text{
    to be chosen later}
\end{equation}
and
\begin{equation}
  \label{2.7}
  \begin{split}
    P_\varphi:&= e^{\delta \sigma_j \varphi}\Bigl(-\frac{d^2}{dx^2} +
    V_j(x)-\tau_j^2\Bigr)e^{-\delta\sigma_j \varphi}\\
&=e^{\delta \sigma_j \varphi} \Bigl(-\frac{d^2}{dx^2} + \sigma_j^2
a(x)^{-2}+V_0(x)-\tau_j^2\Bigr)e^{-\delta\sigma_j\varphi}\\
&=-\frac{d^2}{dx^2} + 2\delta\sigma_j \varphi'\frac{d}{dx} + \sigma_j^2
a(x)^{-2} - \delta^2\sigma_j^2 (\varphi')^2 + \delta\sigma_j \varphi''
+ V_0(x) -\tau_j^2.
  \end{split}
\end{equation}

Using $\Re 2\delta \sigma_j \la \varphi' w', w\ra = -\delta \sigma_j
\la \varphi'' w,w\ra$, we compute
\begin{equation}  \label{2.8}
\begin{split}
    \Re\la P_\varphi w,w\ra_{L^2} &= \|w'\|^2_{L^2} + \Re
    2\delta\sigma_j\la \varphi' w',w\ra_{L^2} \\&\qquad\qquad\qquad+ \la
    (\sigma_j^2(a^{-2}-\delta^2(\varphi')^2) + \delta\sigma_j\varphi''
    + V_0 - \tau_j^2)w,w\ra_{L^2} \\
&=\|w'\|^2_{L^2} + \la (\sigma_j^2(a^{-2}-\delta^2(\varphi')^2)+V_0-\tau_j^2)w,w\ra_{L^2}.
  \end{split}  
\end{equation}
Here and in the sequel we have abbreviated $L^2([x_0,0])$ by $L^2$.
Since $\tau_j^2 \le |V_j(x_0/2)| \le \sigma_j^2 a(x_0/2)^{-2} + |V_0(x_0/2)|$,
we now estimate, on $\supp\,\chi_0$,
\begin{multline}
  \label{2.9}
  \sigma^2_j (a^{-2}-\delta^2(\varphi')^2) + V_0 - \tau_j^2 \ge
  \sigma_j^2\Bigl( \Bigl(\max_{\supp\,\chi_0} a^2\Bigr)^{-1} -
  a(x_0/2)^{-2} - \delta^2 \max_{\supp\,\chi_0} (\varphi')^2\Bigr) \\-
  \max_{\supp\,\chi_0} |V_0| - |V_0(x_0/2)|.
\end{multline}
Because $a'(x)<0$ for $x<0$ and $\supp\,\chi_0\subseteq (x_0/2,0]$, we
can fix $\delta>0$ small enough so that
\[\Bigl(\max_{\supp \,\chi_0} a^2\Bigr)^{-1} -
  a(x_0/2)^{-2} - \delta^2 \max_{\supp\,\chi_0} (\varphi')^2=: \alpha >0.\]
Then if $\sigma_j>1$ is sufficiently large, we can ensure that
\[\sigma_j^2(a^{-2}-\delta^2 (\varphi')^2) + V_0-\tau_j^2 \ge
  \frac{\alpha \sigma_j^2}{2},\quad \text{ on }\supp\,\chi_0.\]
Using this and the fact that $\sigma_j^2>1$ in \eqref{2.8} gives
\[\frac{\alpha}{2}\|w\|^2_{L^2} \le
  \frac{\alpha\sigma^2_j}{2}\|w\|_{L^2}^2 \le \|P_\varphi w\|_{L^2}
  \|w\|_{L^2} \le \frac{1}{\alpha} \|P_\varphi w\|^2_{L^2} +
  \frac{\alpha}{4}\|w\|^2_{L^2}.\]
Therefore
\begin{equation}
  \label{2.10}
  \|w\|^2_{L^2} \le \frac{4}{\alpha^2} \|P_\varphi w\|^2_{L^2}.
\end{equation}

We now use elliptic estimates to show 
\begin{equation}
  \label{2.11}
  \|P_\varphi w\|_{L^2} \lesssim \sigma_j^2 \|\psi_j\|_{L^2}.
\end{equation}
Integration by parts and \eqref{2.3} give
\begin{equation}
  \label{2.12}
  \begin{split}
    \|\psi_j ' \|^2_{L^2} &\le \frac{1}{2} \|\psi_j \|^2_{L^2} +
    \frac{1}{2} \|\psi_j''\|^2_{L^2}\\
&\le \Bigl(\frac{1}{2} + (\max_{x\in [x_0,0]}
V_j(x))^2\Bigr)\|\psi_j\|^2_{L^2}\\
 &\lesssim \sigma^4_j\|\psi_j\|^2_{L^2}.
 \end{split}
\end{equation}
By \eqref{2.3}, \eqref{2.6}, and \eqref{2.7},
\begin{equation}
  \label{2.13}
  \begin{split}
    P_\varphi w &= \chi_0 (P_\varphi e^{\delta \sigma_j\varphi}
    \psi_j) + [P_\varphi,\chi_0] e^{\delta\sigma_j\varphi}\psi_j\\
&=\chi_0 e^{\delta \sigma_j \varphi}\Bigl(-\frac{d^2}{dx^2} + V_j -
\tau_j^2\Bigr)\psi_j + [P_\varphi,\chi_0]e^{\delta\sigma_j
  \varphi}\psi_j\\
&=\Bigl[-\frac{d^2}{dx^2},\chi_0\Bigr]e^{\delta\sigma_j\varphi} \psi_j
+ \Bigl[2\delta\sigma_j \varphi'\frac{d}{dx},\chi_0\Bigr]e^{\delta\sigma_j\varphi}\psi_j.
  \end{split}
\end{equation}
Using the support properties of $\varphi$ and $\chi_0$, we get
\[\Bigl[-\frac{d^2}{dx^2},\chi_0\Bigr]e^{\delta\sigma_j\varphi} =
  \Bigl[-\frac{d^2}{dx^2},\chi_0\Bigr],\quad \Bigl[2\delta \sigma_j
  \varphi'\frac{d}{dx},\chi_0\Bigr]e^{\delta\sigma_j\varphi} = 0.\]
Since $[-\frac{d^2}{dx^2},\chi_0] = -\chi_0'' - 2\chi_0'
\frac{d}{dx}$, using the triangle inequality along with \eqref{2.12}
and \eqref{2.13} establishes \eqref{2.11}.

Since $\varphi \equiv \chi_0\equiv 1$ on $\supp\, (1-\chi)$,
\eqref{2.10} and \eqref{2.11} show
\[e^{2\delta\sigma_j} \int_{\supp\,(1-\chi)} |\psi_j|^2\,dx \le
  \|w\|^2_{L^2} \lesssim \sigma^4_j \|\psi_j\|^2_{L^2},\]
which implies the desired result \eqref{agmon}.
\end{proof}

We now complete the proof by establishing \eqref{quasimode_separated}:
\begin{proposition}\label{prop_quasimode_separated}
  There exists a constant $c$ so that for any $k=0,1,\dots$, we have
\[\Bigl\|\Bigl(-\frac{d^2}{dx^2}+V_j-\tau_j^2\Bigr)u_j\Bigr\|_{H^k((x_0,0))}\le
  C_k e^{-c\sigma_j}.\]
\end{proposition}

\begin{proof}
  As \eqref{agmon} gives
\[\|(1-\chi)\psi_j\|_{L^2} \le e^{-c\sigma_j} \|\psi_j\|_{L^2}\implies
  \|\chi\psi_j\|_{L^2}\ge (1-e^{-c\sigma_j})\|\psi_j\|_{L^2},\]
it suffices to bound the norms of
$(-\frac{d^2}{dx^2}+V_j-\tau_j^2)(\chi \psi_j)$ in terms of
$\|\psi_j\|_{L^2}$.

Integrating by parts gives
\begin{align*}
2\|\chi' \psi_j'\|^2_{L^2} &\le -2\int_{x_0}^0 \Bigl(2 \chi' \chi''
                             \psi_j' +
                             (\chi')^2 \psi_j''\Bigr)\overline{\psi_j}\,dx\\
&\le \|\chi' \psi_j'\|^2_{L^2} + C\int_{\supp\,\chi'} |\psi_j|^2\,dx 
 + C\int_{\supp\,\chi'} \Bigl|V_j(x)-\tau_j^2\Bigr| |\psi_j|^2\,dx.
\end{align*}
Noting that $\supp\,\chi' \subseteq \supp (1-\chi)$, this yields
\begin{equation}
  \label{2.5}
 \|\chi' \psi_j'\|^2_{L^2} \lesssim \sigma_j^2 \int_{\supp\,(1-\chi)} |\psi_j|^2\,dx,
\end{equation}
provided $j$ is sufficiently large.
Using \eqref{2.5} along with \eqref{2.3} and \eqref{agmon}, we get
\begin{align*}
  \Bigl\|\Bigl(-\frac{d^2}{dx^2} + V_j - \tau_j^2\Bigr)\chi \psi_j\Bigr\|_{L^2}
  &= \Bigl\|\Bigl[-\frac{d^2}{dx^2},\chi\Bigr]\psi_j\Bigr\|_{L^2}\\
&\le\|\chi'' \psi_j\|_{L^2} + 2\|\chi' \psi_j'\|_{L^2}\\
&\le C \sigma_j \|\mathbf{1}_{\supp\, (1-\chi)} \psi_j\|_{L^2}\\
&\le e^{-c\sigma_j} \|\psi_j\|_{L^2},
\end{align*}
as desired.

The bounds on the higher Sobolev norms follow by an induction argument
and using \eqref{2.3} and integration by parts repeatedly as above.
\end{proof}

\bibliography{deg_wave}

\end{document}